\documentclass[11pt,final]{article}
\usepackage{authblk}

\usepackage[letterpaper,margin=1in,includefoot,heightrounded]{geometry}
\usepackage[utf8]{inputenc}
\usepackage[english]{babel}

\usepackage{amsmath,amssymb,amsthm,mathrsfs,esint}
\usepackage{graphicx}
\usepackage{enumerate,enumitem}

\usepackage[colorlinks]{hyperref}
\usepackage[notref, notcite]{showkeys}

\renewcommand{\geq}{\geqslant}

\renewcommand{\ge}{\geqslant}
\renewcommand{\le}{\leqslant}

\newcommand{\F}{\mathcal F}
\renewcommand{\L}{\mathcal L}
\renewcommand{\H}{\mathcal H}
\newcommand{\Hd}{\mathcal T}
\newcommand{\Q}{\mathcal Q}
\newcommand{\R}{\mathbb R}
\newcommand{\D}{\mathbb D}
\newcommand{\C}{\mathbb C}
\newcommand{\T}{\mathbb T}
\renewcommand{\Im}{\operatorname{Im}}
\renewcommand{\Re}{\operatorname{Re}}

\renewcommand{\Omega}{\varOmega}

\theoremstyle{plain}
\newtheorem{theorem}{Theorem}
\newtheorem{lemma}[theorem]{Lemma}

\newtheorem*{thm*}{Theorem}

\theoremstyle{remark}
\newtheorem*{remark*}{Remark}

\title{Overhanging and touching waves \protect\\in constant vorticity flows}

\author[1]{Vera~Mikyoung~Hur\thanks{E-mail:~verahur@math.uiuc.edu}}
\author[2]{Miles~H.~Wheeler\thanks{E-mail:~mw2319@bath.ac.uk}}
\affil[1]{Department of Mathematics, University of Illinois at Urbana-Champaign \protect\\ Urbana, IL 61801, USA}
\affil[2]{Department of Mathematical Sciences, University of Bath \protect\\ Bath BA2 7AY, UK}

\begin{document}

\maketitle

\begin{abstract}
We show the existence of periodic traveling waves at the free surface of a two dimensional, infinitely deep, and constant vorticity flow, under gravity, whose profiles are overhanging, including one which intersects itself to enclose a bubble of air. Numerical evidence has long suggested such overhanging and touching waves, but a rigorous proof has been elusive. \mbox{Crapper's} celebrated capillary waves in an irrotational flow have recently been shown to yield an exact solution to the problem for zero gravity, and our proof uses the implicit function theorem to construct nearby solutions for weak gravity.
\end{abstract}

\section{Introduction}\label{sec:intro}

We consider periodic traveling waves at the free surface of an incompressible inviscid fluid in two dimensions, under gravity, without the effects of surface tension. When the flow is irrotational, the wave profile is necessarily the graph of a single-valued function \cite{spielvogel} (see also \cite{CSV:constant-vorticity}). In constant vorticity flows, by contrast, numerical investigations (see, for instance, \cite{ss:deep, sp:steep, DH1, DH2}) have revealed profiles with multi-valued height and even profiles which intersect themselves tangentially above the trough to enclose a bubble of air. Constantin, Strauss and V\u{a}rv\u{a}ruc\u{a} \cite{CSV:constant-vorticity} conjectured that such {\em overhanging} and {\em touching} waves indeed exist.
Here we give a proof of this conjecture.

Crapper \cite{crapper} discovered a remarkable family of exact solutions to the capillary wave problem---that is, nonzero surface tension and zero gravity---in an irrotational flow, whose profiles become more rounded as the amplitude increases, opposite to gravity waves, so that overhanging profiles appear, limited by a touching wave. See Figure~\ref{fig:crapper}. Akers, Ambrose and Wright \cite{AAW:overhanging} then employed a perturbation method to construct nearby solutions for sufficiently weak gravity, and in particular overhanging capillary-gravity waves. C\'{o}rdoba, Enciso and Grubic \cite{CEG:touching} took matters further and constructed a touching wave. 
Recently, the authors \cite{HW:rapids} (see also \cite{HVB}) showed that Crapper's capillary waves also give the profiles of periodic traveling waves in constant vorticity flows, without the effects of gravity and surface tension. We follow a perturbation argument, similar to \cite{AAW:overhanging, CEG:touching, ce:existence} and others, to construct overhanging and touching waves for nonzero gravity. 

But it is the rotational effect which generates overhanging and touching profiles for our problem, rather than the capillary effect \cite{AAW:overhanging, CEG:touching, ce:existence}. Although the unperturbed fluid surface is the same as Crapper's wave, the fluid flow beneath the surface is completely different, and so are the governing equations. See \cite{HW:rapids} for more discussion, and also see \cite{bp:stability} for a study of the stability of the solutions with zero gravity.

For the capillary wave problem, Okamoto and Sh\={o}ji \cite{OS:paper, OS:book} produced closed-form recurrence relations among the Fourier coefficients for the linearized operator about Crapper's wave, which enabled \cite{AAW:overhanging, CEG:touching, ce:existence} and others to work out their perturbation arguments. Unfortunately, such relations seem unwieldy for nonzero constant vorticity and zero surface tension. Instead we reformulate our problem for a holomorphic function in the unit disk, for which the zero-gravity exact solution is given as a rational function (see \eqref{def:w(A)}) and the commutator, first introduced in \cite{BDT1, BDT2} for zero vorticity, and its linearized operator can be evaluated by means of the calculus of residues (see \eqref{def:Q} and \eqref{def:Qw}). The novelty of our approach is that, to establish the invertibility of the linearized operator, we relate it the problem of finding holomorphic solutions to a complex ODE with meromorphic coefficients. In retrospect, this technique seems quite natural, but so far we have been unable to find other examples of its use in the literature. We believe that similar methods could be applied to a much wider range of fluids problems which possess explicit solutions given in terms of conformal mappings.

We begin in Section~\ref{sec:prelim} by stating the problem and the results. In Section~\ref{sec:reformulation} we reformulate the problem in conformal coordinates and, in turn, for a holomorphic function in the unit disk. In Section~\ref{sec:IFT} we employ the implicit function theorem to prove our results. 
Section~\ref{sec:extensions} discusses how one can possibly take matters further to finite depth, point vortices, and hollow vortices, among others. We pause to remark that Crapper's waves also make exact solutions for point vortices, rather than constant vorticity, without the effects of gravity or surface tension \cite{CR:pointvortex}. Crowdy and his collaborators \cite{crowdy:hollow-vortex, WC:hollow-vortex, CNK;hollow-vortex} discovered exact solutions for non-rotating hollow vortices with the effects of surface tension and showed that, interestingly, the same conformal mapping makes exact solutions for rotating hollow vortices with $N$-fold symmetry. Appendix~\ref{sec:appn} gives a summary of \cite{HW:rapids} and, importantly, corrects errors in \cite{HW:rapids}. 
 
\section{Preliminaries and the statement of the results}\label{sec:prelim}

\subsection{Stream function formulation}\label{sec:stream}

We consider a two dimensional, infinitely deep, and constant vorticity flow of an incompressible inviscid fluid, under gravity, without the effects of surface tension, and periodic traveling waves at the fluid surface. We assume for simplicity that the fluid has the unit density. Suppose for definiteness that in Cartesian coordinates, waves propagate in the $x$ direction and gravity acts in the negative $y$ direction. In a frame of reference moving with a constant velocity, suppose that the fluid flow is stationary and occupies a region $D$ in the $(x,y)$ plane, bounded above by a free surface~$S$.

Let $\psi(x,y)$ denote a stream function so that $(\psi_y,-\psi_x)$ is the velocity of the fluid, and $\psi$ satisfies
\begin{subequations}\label{eqn:stream}
\begin{alignat}{2}
&\nabla^2\psi=-\omega && \text{in $D$}, \label{eqn:vorticity} \\ 
&\psi=0 &&  \text{on $S$}, \label{eqn:kinematic bc} \\ 
&\tfrac12|\nabla\psi|^2+gy=b && \text{on $S$}, \label{eqn:dynamic bc}\\
&\nabla\psi-(0,-\omega y-c)\to(0,0) \qquad && \text{as $y\to-\infty$} \label{eqn:infty bc}
\end{alignat}
\end{subequations}
for some $c>0$, the wave speed. Here $\omega$ denotes the vorticity and we assume that it is constant in $D$. Note that $S$ is a free boundary, and \eqref{eqn:kinematic bc} and \eqref{eqn:dynamic bc} are the kinematic and dynamic boundary conditions, where $g\geq0$ is the gravitational constant, and $b\in\mathbb{R}$ the Bernoulli constant. For zero vorticity, that is, $\omega=0$, \eqref{eqn:infty bc} expresses that there is no motion of the fluid at the infinite bottom. 
Additionally we assume that $D$ and $\psi$ are $2\pi/k$ periodic in the $x$ direction for some wave number $k>0$, and symmetric about the vertical lines below the crest and the trough.

Introducing dimensionless variables\footnote{Rather than introducing new notation for all the variables, we choose to write, for instance, $x\mapsto kx$. This is to be read `$x$ is replaced by $kx$', so that hereafter the symbol $x$ will mean a dimensionless variable.} 
\[
x\mapsto kx, \qquad y\mapsto ky, \qquad \psi \mapsto (k/c)\psi,
\]
and dimensionless parameters
\[
\Omega=\omega/ck,\qquad G=g/kc^2,\qquad B=b/c^2,
\]
we can rewrite \eqref{eqn:stream} more conveniently as
\begin{equation}\label{eqn:stream0}
\begin{aligned}
&\nabla^2\psi=-\Omega && \text{in $D$}, \\ 
&\psi=0 && \text{on $S$}, \\ 
&\tfrac12|\nabla\psi|^2+Gy=B && \text{on $S$}, \\
&\nabla\psi-(0,-\Omega y-1)\to(0,0)\qquad && \text{as $y\to-\infty$}.
\end{aligned}
\end{equation}
Suppose that the fluid surface is given parametrically as
\begin{equation}\label{def:S}
S=\{(x(\alpha),y(\alpha)): \alpha\in\mathbb{R}\}.
\end{equation}
The periodicity and symmetry conditions become
\begin{equation}\label{eqn:symm(z)}
  \begin{aligned}
  &x(\alpha+2\pi)=x(\alpha)+2\pi&&\text{and}\quad y(\alpha+2\pi)=y(\alpha),\\
  &x(-\alpha)=-x(\alpha)&&\text{and}\quad y(-\alpha)=y(\alpha) 
  \end{aligned}
\end{equation}
for all $\alpha\in\R$, and 
\[
\psi(-x,y)=\psi(x,y)=\psi(x+2\pi,y)\qquad \text{for all $(x,y)\in D$}.
\]

\subsection{Exact solution for zero gravity}\label{sec2:crapper}

In what follows, we identify $\R^2$ with $\C$ whenever it is convenient to do so and employ the notation $z=x+iy$. 

For zero gravity, that is, $G=0$, the authors \cite{HW:rapids} recently showed that 
\begin{equation}\label{def:z(A)}
z(\alpha;A):=\alpha-\frac{4iAe^{-i\alpha}}{1+Ae^{-i\alpha}},
\end{equation}
together with
\begin{equation}\label{def:omega(A)} 
\Omega(A):=\frac{1-A^2}{1-3A^2}\quad\text{and}\quad 
B(A):=\frac12\left(\frac{1+A^2}{1-3A^2}\right)^2,
\end{equation}
make an exact solution to \eqref{eqn:stream0}--\eqref{eqn:symm(z)} for an appropriate stream function (see \eqref{eqn:psi}), depending on the real parameter $A$. 
But, unfortunately, there were sign errors, among others, causing some equations in \cite{HW:rapids} to appear incorrect. In Appendix~\ref{sec:appn} we detail how to correct these errors. 

Surprisingly, the same fluid surface also makes Crapper's exact solution to the capillary wave problem (nonzero surface tension and zero gravity) in an irrotational flow for an appropriate value of the surface tension coefficient, depending on $A$ \cite{crapper}. See also \cite{HW:rapids}.

\begin{figure}
    \centering
    \includegraphics[scale=1]{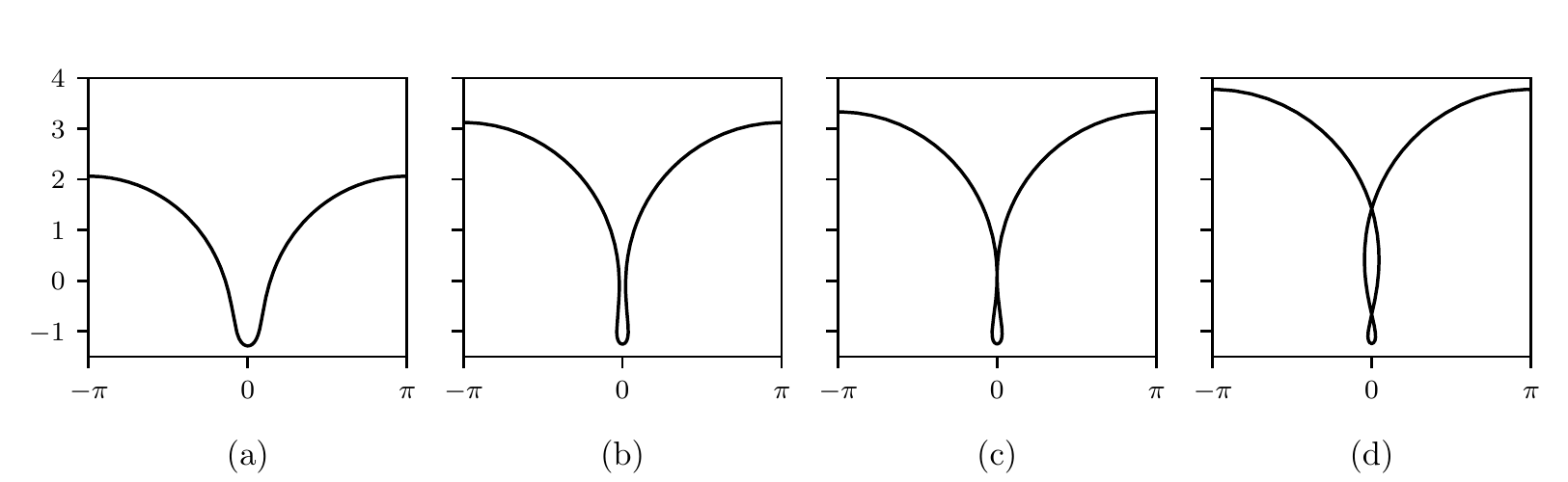}
    \caption{The profiles of \eqref{def:z(A)} in the $(x,y)$ plane for four values of $A$. (a)~$A=0.8A_{\max} < \sqrt{2}-1$ (see \eqref{def:Amax}), and the profile is not overhanging. 
    (b)~$A=0.97A_{\max} > \sqrt{2}-1$, and it is overhanging. 
    (c)~$A=A_{\max}$. The profile intersects itself tangentially at one point of the trough line. 
    (d)~$A=1.06A_{\max}$, and it intersects itself transversely at two points of the trough line. The fluid surfaces in the panels (a)--(c) give rise to physical solutions to \eqref{eqn:stream0}--\eqref{eqn:symm(z)} but (d) does not. 
    }
    \label{fig:crapper}
\end{figure}

Figure~\ref{fig:crapper} shows the profiles of \eqref{def:z(A)} for four values of $A$ in the $(x,y)$ plane in the range $x\in[-\pi,\pi]$. By symmetry, it suffices to take $A\geq0$. When $A$ is small, the fluid surface is not overhanging, that is, $y$ can be given as a function of $x$. See, for instance, Figure~\ref{fig:crapper}(a). When $A$ is sufficiently large, on the other hand, the fluid surface intersects itself transversely at two points of the trough line and, hence, the fluid flow becomes multi-valued, giving rise to a physically `unrealistic' solution to \eqref{eqn:stream0}--\eqref{eqn:symm(z)}. See, for instance, Figure~\ref{fig:crapper}(d). A straightforward calculation reveals that the profile of \eqref{def:z(A)} is not overhanging as long as
\begin{equation}\label{eqn:A=overhang}
  A<\sqrt{2}-1 = 0.4142135623730950\ldots.
\end{equation}
Recall \cite{crapper} that it does not intersect itself so long as
\begin{equation}\label{eqn:A=touching}
A<A_{\max}=0.4546700164520109\dots,
\end{equation}
where 
\begin{equation}\label{def:Amax}
A_{\max}:=\max_{\alpha\in[-\pi,\pi]}\left(\frac{2\sin\alpha}{\alpha}-\cos\alpha\right)
-\sqrt{\left(\max_{\alpha\in[-\pi,\pi]}\left(\frac{2\sin\alpha}{\alpha}-\cos\alpha\right)\right)^2-1}.
\end{equation}
When $\sqrt{2}-1<A<A_{\max}$, the profile of \eqref{def:z(A)} is overhanging but does not intersect itself. See, for instance, Figure~\ref{fig:crapper}(b). When $A=A_{\max}$, it intersects itself tangentially at one point of the trough line, enclosing a bubble of air, namely, a touching wave. See Figure~\ref{fig:crapper}(c). 

In what follows we restrict the attention to $A \in [0,1/2)$, which includes all the physical solutions ($A \le A_{\max}$) and some nonphysical solutions ($A > A_{\max}$). We remark that both the dimensionless vorticity parameter $\Omega(A)$ and the crest-to-trough vertical distance $\Im(z(\pi;A)-z(0;A))$ are strictly increasing with $A\in[0,1/2)$.

\subsection{Statement of the results}\label{sec:result}

The aim here is to construct solutions to \eqref{eqn:stream0}--\eqref{eqn:symm(z)} nearby \eqref{def:z(A)} and \eqref{def:omega(A)} for small but nonzero values of $G$, particularly, overhanging and touching waves.

Below we state our results.

\begin{theorem}[Overhanging waves]\label{thm:overhang}
For each $A \in (\sqrt 2-1,A_{\max})$ (see \eqref{def:Amax}) and $G$ sufficiently small, \eqref{eqn:stream0}--\eqref{eqn:symm(z)} has a solution, where $\Omega=\Omega(A)$ and $B=B(A)$ are in \eqref{def:omega(A)}, whose fluid surface does not intersect itself but is overhanging, that is, $y$ cannot be given as a function of $x$.
\end{theorem}

\begin{theorem}[Touching waves]\label{thm:touching}
For $G$ sufficiently small, there exists a solution to \eqref{eqn:stream0}--\eqref{eqn:symm(z)} for which $S$ intersects itself tangentially along the trough line, enclosing a bubble of air.
\end{theorem}

In Theorem~\ref{thm:touching}, $\Omega=\Omega(A)$ and $B=B(A)$ for $A\approx A_{\max}$. 

We emphasize that Theorems~\ref{thm:overhang} and \ref{thm:touching} are the first to rigorously establish that overhanging and touching profiles exist for surface gravity waves. There is persuasive numerical evidence (see, for instance, \cite{DH1,DH2,DH3} and references therein) of their existence. Also there is a global bifurcation result \cite{CSV:constant-vorticity} which allows for overhanging profiles, although the result is incapable of determining whether such profiles actually exist. For zero vorticity, that is, $\omega=0$, by contrast, overhanging waves cannot exist. See, for instance, \cite{CSV:constant-vorticity} for more discussion. 

The proof of Theorems~\ref{thm:overhang} and \ref{thm:touching} is based on the exact solution for $G=0$, discussed in Section~\ref{sec2:crapper}, and uses the implicit function theorem to construct nearby solutions for small $G$. The same strategy has been implemented for the existence of overhanging and touching capillary-gravity waves (see, for instance, \cite{AAW:overhanging, CEG:touching, ce:existence}), based instead on Crapper's exact solution to the capillary wave problem in an irrotational flow. Although the zero-gravity fluid surface is the same, the physical problem here is completely different from the capillary-gravity wave problem. Actually, the fluid flows are completely different. See \cite{HW:rapids} for more discussion. The linearized operator of the capillary wave problem about Crapper's wave was treated in \cite{OS:paper,OS:book}, examining closed-form recurrence relations among the Fourier coefficients. But such an approach seems unwieldy for our problem. We develop an alternative approach, relating the linearized operator to a complex ODE with meromorphic coefficients, whose solvability can be studied by means of the calculus of residues. 

\section{Reformulation}\label{sec:reformulation}

\subsection{Reformulation via conformal mapping}\label{sec:conformal}

We introduce 
\begin{equation}\label{def:conformal}
z=z(\alpha+i\beta),
\end{equation}
which maps $\R\times(-\infty,0)$ to $D$ conformally, $\mathbb{T}\times(-\infty,0)$ to $(\mathbb{T}\times\R)\cap D$, and satisfies 
\[
z(\alpha+i\beta)-(\alpha+i\beta)\to0\qquad\text{as $\beta\to-\infty$}.
\]
Suppose that \eqref{def:conformal} extends to map $\mathbb{R}\times(-\infty,0]$ to $D\cup S$ continuously. This allows us to reformulate \eqref{eqn:stream0} in `conformal coordinates' as
\begin{equation}\label{eqn:bernoulli}
(1+\Omega(y+y\H y_\alpha-\H(yy_\alpha)))^2=(B-2Gy)((1+\H y_\alpha)^2+y_\alpha^2)
\qquad \text{for $\beta=0$}.
\end{equation}
Here $\H$ denotes the periodic Hilbert transform: for instance, for $y\in L^2(\mathbb{T})$, 
\begin{equation}\label{def:H}
\H y(\alpha)=\frac1{2\pi}\operatorname{PV}\int_\T y(\alpha')\cot\left(\frac{\alpha-\alpha'}2\right)~ d\alpha',
\end{equation}
where $\operatorname{PV}$ stands for Cauchy's principal value integral. Alternatively, 
\begin{equation}\label{def:H'}
\H e^{in\alpha}=-i\operatorname{sgn}(n)e^{in\alpha},\qquad n\in\mathbb{Z}.
\end{equation}
See, for instance, \cite{DH1,DH2} for details. 
Here and in what follows, we regard $y$ as a real-valued function of $\alpha\in\mathbb{R}$ whenever it is convenient to do so. Throughout we use subscripts for partial derivatives and primes for variables of integration. We pause to remark that $z(\alpha)=\alpha+(\H+i)y(\alpha)$. In other words, $S=\{(\alpha+\H y(\alpha), y(\alpha)):\alpha\in\R\}$. 

For zero vorticity, that is, $\Omega=0$, \eqref{eqn:bernoulli} becomes
\[
1=(B-2Gy)(y_\alpha^2+y_\beta^2)\qquad \text{for $\beta=0$},
\]
and indeed \eqref{eqn:stream0} can be formulated as (local) elliptic boundary value problem for $y$ in a fixed domain. For nonzero vorticity, on the other hand, the nonlocal term $\H(yy_\alpha)$ causes technical difficulties. But the commutator formula \cite{BDT1,BDT2} 
\begin{equation}\label{eqn:Q(a)}
  (y\H y_\alpha-\H(yy_\alpha))(\alpha)=\frac1{8\pi}\int_\T
  (y(\alpha)-y(\alpha'))^2\csc^2\left(\frac{\alpha-\alpha'}2\right)~d\alpha'
\end{equation}
turns out to be instrumental. See \eqref{def:Q} and \eqref{def:Qw}. 

A (smooth) solution of \eqref{eqn:bernoulli} gives rise to a solution of \eqref{eqn:stream0}, provided that 
\begin{gather}
\text{$\alpha\mapsto z(\alpha)$ is injective for all $\alpha\in\R$ \label{eqn:injective}}
\intertext{and}
\text{$z_\alpha(\alpha) \neq 0$ for all $\alpha \in \R$}.\label{eqn:stagnation}
\end{gather}
We refer the reader to, for instance, \cite{DH1,DH2} for details. 

Recall \cite{DH1,DH2} that \eqref{eqn:injective} states that the fluid surface does not intersect itself, while \eqref{eqn:stagnation} ensures that \eqref{def:conformal} is well-defined throughout $\R\times(-\infty,0]$. There is numerical evidence \cite{DH1,DH2} (see also \cite{DH3}) that solutions of \eqref{eqn:bernoulli} can be found even when \eqref{eqn:injective} fails to hold, although such solutions would be nonphysical because the fluid surface intersects itself and the fluid flow becomes multi-valued. Such nonphysical solutions can nevertheless be useful, and indeed we will make use of them to construct a touching wave for nonzero gravity. 

When \eqref{eqn:stagnation} fails to hold, on the other hand, the fluid surface develops a stagnation point, where the velocity of the fluid vanishes in the moving frame of reference. There is numerical evidence \cite{DH1,DH2} that for any value of $\Omega$, the solutions of \eqref{eqn:bernoulli} are ultimately limited by an `extreme' wave (in an appropriate function space, for instance, the (amplitude) $\times$ (wave speed) plane), whose profile has a stagnation point at the crest, enclosing a $120^\circ$ angle. 
Extreme waves are beyond the scope of this work, and throughout we will require that \eqref{eqn:stagnation} holds true. 

\subsection{Reformulation for holomorphic functions in the unit disk}\label{sec:w(zeta)}

It is convenient to introduce
\begin{equation}\label{def:zeta}
\zeta = e^{-i(\alpha+i\beta)}, 
\end{equation}
which maps $\T\times (-\infty,0)$ to $\D:=\{\zeta\in\C:|\zeta|<1\}$ and $\T\times\{0\}$ to $\partial\D$, by \eqref{eqn:symm(z)}. Abusing notation, we write \eqref{def:conformal} as $z(\zeta)$, and let 
\begin{equation}\label{def:w}
z(\zeta)=i\log\zeta+w(\zeta).
\end{equation}
Note from \eqref{def:conformal} and \eqref{eqn:symm(z)} that $w$ is a single-valued holomorphic function of $\zeta\in\D$, although $z$ is not, and that 
\[
w(\overline\zeta)=-\overline{w(\zeta)}.
\]
We can then rewrite \eqref{eqn:bernoulli} as 
\begin{equation}\label{eqn:w}
\frac12\frac{(1+\Omega(\Im w+\Q(w)))^2}{|1-i\zeta w_\zeta|^2}=B-G\Im w
\qquad \text{for $|\zeta|=1$},
\end{equation}
where after the change of variables, \eqref{eqn:Q(a)} becomes
\begin{equation}\label{def:Q}
\Q (w(\zeta)):=-\frac{\zeta}{2\pi i}\ointctrclockwise_{|\zeta'|=1}
\left(\frac{\Im(w(\zeta)-w(\zeta'))}{\zeta-\zeta'}\right)^2~d\zeta'
\qquad\text{for $|\zeta|=1$}.
\end{equation}
Of particular usefulness for our purpose is that when $w$ is meromorphic in the unit disk, one can evaluate the right hand side of \eqref{def:Q} by means of the calculus of residues. See Section~\ref{sec3:crapper}. Note that \eqref{eqn:stagnation} becomes $|1-i\zeta w_\zeta|^2=|z_\alpha|^2\neq0$ for $|\zeta|=1$. 

\subsection{The exact solution revisited}\label{sec3:crapper}

For zero gravity, that is, $G=0$, we deduce from Sections~\ref{sec2:crapper} and \ref{sec:w(zeta)} that, for any $A\in[0,1/2)$, 
\begin{equation}\label{def:w(A)}
w(\zeta;A):=-\frac{4iA\zeta}{1+A\zeta}
\end{equation}
together with \eqref{def:omega(A)} are an exact solution of \eqref{eqn:w}. Clearly, $w(A)$ is holomorphic in $\D$ and $w(\overline{\zeta};A)=-\overline{w(\zeta;A)}$. 

When $\sqrt{2}-1<A<A_{\max}$, where $A_{\max}$ is in \eqref{def:Amax}, the fluid surface given by \eqref{def:w}, where $w$ is in \eqref{def:w(A)}, is overhanging but does not intersect itself. When $A=A_{\max}$, it intersects itself tangentially at one point over the period and can make sense of a solution of \eqref{eqn:stream0}--\eqref{eqn:symm(z)}. When $A>A_{\max}$, on the other hand, it intersects itself transversely at two points, giving rise to a nonphysical solution of \eqref{eqn:stream0}--\eqref{eqn:symm(z)}.

We give some details on how \eqref{def:w(A)} solves \eqref{eqn:w}, for the sake of completeness and also for future reference. We restrict attention to $A\in[0,1/2)$. Since $\overline{\zeta}=1/\zeta$ for $|\zeta|=1$, 
\begin{equation}\label{eqn:Im(w(A))}
\Im w(\zeta;A)=2\left(-1+\frac{A}{\zeta+A}+\frac{1/A}{\zeta+1/A}\right)\qquad \text{for $|\zeta|=1$}.
\end{equation}
Substituting \eqref{eqn:Im(w(A))} into \eqref{def:Q}, we evaluate the integral by means of the calculus of residues to arrive at 
\[
\Q(w(\zeta;A))=-8(A-1/A)^{-2}\left(\frac{A}{\zeta+A}-\frac{1/A}{\zeta+1/A}\right)
\]
for $|\zeta|=1$, whence
\begin{equation}\label{eqn:crapper1}
1+\Omega(A)(\Im w(\zeta;A)+\Q(w(\zeta;A)))=(1-2\Omega(A))\frac{(\zeta-A)(\zeta-1/A)}{(\zeta+A)(\zeta+1/A)}
\end{equation}
for $|\zeta|=1$. On the other hand, 
\begin{equation}\label{eqn:crapper2}
  1-i\zeta w_{\zeta}(\zeta;A)=\left(\frac{\zeta-1/A}{\zeta+1/A}\right)^2,
\end{equation}
whence
\begin{equation}\label{eqn:|z'(A)|}
|1-i\zeta w_\zeta(\zeta;A)|^2=\left(\frac{(\zeta-A)(\zeta-1/A)}{(\zeta+A)(\zeta+1/A)}\right)^2 \qquad\text{for $|\zeta|=1$}.
\end{equation}
A straightforward calculation reveals that \eqref{eqn:w} holds true for \eqref{def:omega(A)} and $G=0$. Note that $|1-i\zeta w_\zeta(A)|^2\neq0$ for $|\zeta|=1$. 

\section{Proof of Theorems~\ref{thm:overhang} and \ref{thm:touching}}\label{sec:IFT}

For $a\in(0,1)$ and fixed, we define the Banach spaces
\begin{align*}
X&=\{w \in C^{3+a}(\D,\C):\text{$w$ is holomorphic in $\D$ and $w(\overline{ \zeta})=-\overline{w(\zeta)}$}\},\\
Y&=\{f\in C^{2+a}(\partial\D,\R):f(\overline{\zeta})=f(\zeta)\}
\intertext{and the open set}
U&=\{w\in X:\text{$1-i\zeta w_\zeta(\zeta)\neq0$ for $|\zeta|=1$}\} \subset X.
\end{align*}
We define $\F\colon U\times\R^2 \to Y$ as
\[
\F(w;G,A)=\frac12\frac{(1+\Omega(A)(\Im w+\Q (w)))^2}{|1-i\zeta w_\zeta|^2}+G\Im w-B(A)
\qquad \text{for $|\zeta|=1$},
\]
where $\Omega(A)$ and $B(A)$ are in \eqref{def:omega(A)}. Since $\overline{\zeta}=1/\zeta$ for $|\zeta|=1$, note from \eqref{def:Q} that $\Q(w(\overline{\zeta}))=\Q(w(\zeta))$ for $w \in X$. Since \eqref{eqn:Q(a)} is real analytic and \eqref{def:zeta} is holomorphic, $\Q$ is real analytic. Therefore $\F$ is well-defined and real analytic. 

Recall from Section~\ref{sec3:crapper} that for any $A\in[0,1/2)$,
\begin{equation}\label{eqn:F(A)=0}
\F(w(A);0,A)=0,
\end{equation}
and our task is to construct nearby solutions of $\F(w;G,A)=0$ for $G\neq0$ small. 

\begin{theorem}\label{thm:IFT}
For each $A_0\in(0,1/2)$ there exists $\varepsilon > 0$ and a real-analytic map 
\[
W\colon (-\varepsilon,\varepsilon) \times (A_0-\varepsilon,A_0+\varepsilon) \to U
\]
such that $W(0,A)=w(A)$ and 
\[
\F(W(G,A);G,A)=0.
\]
Moreover, there exists $\delta > 0$ such that $W(G,A)$ is the unique solution of $\F(w;G,A)=0$ for which $\|w-w(A)\|_X < \delta$.
\end{theorem}

Our proof of Theorem~\ref{thm:IFT} is based on the implicit function theorem. For $w\in U$, $G\in\R$ and $A\in[0,1/2)$, we calculate that the linearized operator $\F_w(w;G,A)\colon X\to Y$ is given as
\begin{align*}
\F_w(w;G,A)v=&\frac{1+\Omega(A)(\Im w+\Q (w))}{|1-i\zeta w_\zeta|^2}\Omega(A)(\Im v+\Q_w(w)v)\\
&-\frac{(1+\Omega(A)(\Im w+\Q(w)))^2}{|1-i\zeta w_\zeta|^4}\Im(\overline{(1-i\zeta w_\zeta)}\zeta v_\zeta)+G\Im v\qquad\text{for $|\zeta|=1$}.
\end{align*}
Recalling \eqref{eqn:crapper1}, \eqref{eqn:crapper2} and \eqref{eqn:|z'(A)|}, after some algebra we arrive at
\begin{align*}
\F_w(w(A);0,A)v=&(1-2\Omega(A))\frac{(\zeta+A)(\zeta+1/A)}{(\zeta-A)(\zeta-1/A)}\Omega(A)(\Im v+\Q_w(w(A))v)\\
&-\left((1-2\Omega(A))\frac{(\zeta+A)(\zeta+1/A)}{(\zeta-A)(\zeta-1/A)}\right)^2
\Im\bigg(\left(\frac{\zeta-A}{\zeta+A}\right)^2\zeta v_\zeta\bigg)
\end{align*}
for $|\zeta|=1$. On the other hand, linearizing \eqref{def:Q} about $w(A)$,
\begin{align*}
\Q_w(w(A))v(\zeta)&=-\frac{\zeta}{\pi i}\ointctrclockwise_{|\zeta'|=1}
\frac{\Im(w(\zeta;A)-w(\zeta';A))\Im(v(\zeta)-v(\zeta'))}{(\zeta-\zeta')^2}~d\zeta'\\
&=\Im\left(-\frac{\zeta}{\pi i}\ointctrclockwise_{|\zeta'|=1}
\frac{\Im(w(\zeta;A)-w(\zeta';A))}{(\zeta-\zeta')^2}(v(\zeta)-v(\zeta'))~d\zeta'\right)
\end{align*}
for $|\zeta|=1$. Recalling \eqref{eqn:Im(w(A))}, since
\[
\frac{\Im(w(\zeta;A)-w(\zeta';A))}{(\zeta-\zeta')^2}=
\frac{2}{\zeta-\zeta'}\left(\frac{A}{(\zeta+A)(\zeta'+A)}+\frac{1/A}{(\zeta+1/A)(\zeta'+1/A)}\right)
\]
has poles at $\zeta'=\zeta$ and $-A$ in $\D$, we evaluate the integral by means of the calculus of residues to arrive at a strikingly simple, albeit non-local, formula
\begin{equation}\label{def:Qw}
\Q_w(w(A))v(\zeta)=\Im\left(-4A\zeta\frac{v(\zeta)-v(-A)}{(\zeta+A)^2}\right)
\qquad\text{for $|\zeta|=1$}.
\end{equation}
Therefore
\begin{equation}\label{def:F'}
\F_w(w(A);0,A)v=-(1-2\Omega(A))\frac{(\zeta+A)(\zeta+1/A)}{(\zeta-A)(\zeta-1/A)}\L(A)v \qquad\text{for $|\zeta|=1$},
\end{equation}
where
\begin{equation}\label{def:L}
\L(A)v:=\Im\left((1-2\Omega(A))\zeta\frac{(\zeta-A)(\zeta+1/A)}{(\zeta+A)(\zeta-1/A)}v_\zeta
-\Omega(A)v+4A\Omega(A)\zeta\frac{v-v(-A)}{(\zeta+A)^2}\right).
\end{equation}
We pause to remark that $\L(A)$ is a `generalized' Riemann--Hilbert operator, with the difference that in addition to $v$ itself, $v_\zeta$ appears as well as the non-local term $v(-A)$. General classes of Riemann--Hilbert operators including those of the form \eqref{def:L} with H\"older continuous coefficients are discussed in \cite[Section~34]{gakhov} and \cite[Section~71]{muskhelishvili}, among many others, and $\L(A)v=f$ can be transformed into various integral equations. We emphasize that the particularly simple form of \eqref{def:L} comes from the use of the commutator formula \eqref{def:Q}, introduced by \cite{BDT1,BDT2} for zero vorticity.

\begin{lemma}\label{lem:Fredholm}
For any $A\in[0,1/2)$, $\L(A) \colon X \to Y$ is Fredholm with index zero.
\end{lemma}

\begin{proof}
The principal part of $\L(A)$ is 
\begin{equation}\label{def:L0}
\L_0(A)v:=\Im\left((1-2\Omega(A))\zeta\frac{(\zeta-A)(\zeta+1/A)}{(\zeta+A)(\zeta-1/A)} v_\zeta\right),
\end{equation}
which is a classical Riemann--Hilbert operator acting on $v_\zeta$. Since 
\begin{equation*}
\inf_{\zeta\in\partial\D}\left|(1-2\Omega(A))\zeta\frac{(\zeta-A)(\zeta+1/A)}{(\zeta+A)(\zeta-1/A)}\right|>0
\end{equation*}
whenever $A\in[0,1/2)$, classical elliptic theory (see, for instance, \cite[Chapter~1, Section~2]{volpert:book}) implies that 
\[
\|v_\zeta\|_{C^{2+a}(\D)}\le C(\|\L_0(A)v\|_{C^{2+a}(\partial\D)}+\|v\|_{C^0(\D)})
\]
for some constant $C$ independent of $v$ and, hence, by interpolation,  
\begin{equation}\label{eqn:schauder}
\|v\|_X\le C(A)(\|\L(A)v\|_Y+\|v\|_{C^0(\D)}).
\end{equation}
(Alternatively, one can obtain \eqref{eqn:schauder} by reformulating \eqref{def:L} as an elliptic system for the real and imaginary parts of $v$ and verifying that it satisfies the hypotheses of \cite{ADN}.)
Since $X$ is compact in $C^0(\overline\D)$, we conclude that $\L(A)\colon  X\to Y$ is semi-Fredholm, with closed range and a finite-dimensional kernel. It remains to show that the index of $\L(A)$ is zero.
  
Since the index of $\L(A)$ is continuous and, hence, independent of $A\in[0,1/2)$, it suffices to show that the index of $\L(0)$ is zero. We write $v\in X$ as 
\[
v(\zeta) = \sum_{n=0}^\infty v_n \zeta^n,
\]
where the coefficients $v_n$ are purely imaginary by symmetry, so that
\[
\L(0)v=\Im(\zeta v_\zeta-v)=\Im\left(\sum_{n=0}^\infty(n-1)v_n \zeta^n\right)
\]
by \eqref{def:L} and \eqref{def:omega(A)}. A straightforward calculation reveals that the kernel of $\L(0)$ is one dimensional, spanned by $i\zeta$, and the co-kernel is one dimensional, spanned by $\Re\zeta$. Therefore $\L(0):X\to Y$ is Fredholm with index zero, and the proof is complete.
\end{proof}

Therefore $\L(A)\colon X\to Y$ is invertible provided that the kernel is trivial. We will study the kernel of $\L(A)$ by relating it to a complex ODE, for which the following is useful.

\begin{lemma}\label{lem:elementary}
  Suppose that $v$ is holomorphic in a neighborhood of $\zeta_0\in\C$ and that
    \begin{equation}\label{eqn:v'+pv=q}
      v_\zeta+pv=q,
    \end{equation}
    where $p$ and $q$ are meromorphic with at most simple poles at $\zeta_0$.
  \begin{enumerate}[label=\rm(\roman*)]
  \item\label{lem:elementary:basic} $\operatorname{Res}(q,\zeta_0) = v(\zeta_0) \operatorname{Res}(p,\zeta_0)$. 
  \item\label{lem:elementary:qzero}
    If $q \equiv 0$ and $v \not\equiv 0$ then $\operatorname{Res}(p,\zeta_0) \le 0$ is an integer.
  \item\label{lem:elementary:q1}
    If $\operatorname{Res}(p,\zeta_0)=-2$, so that 
    \begin{align}
      \label{eqn:p0p1q0q1}
      p(\zeta) = \frac{-2}{\zeta-\zeta_0} + p_{0} + p_1(\zeta-\zeta_0) + \cdots
      \quad\text{and}\quad 
      q(\zeta) = \frac{-2v(\zeta_0)}{\zeta-\zeta_0} + q_{0} + q_1(\zeta-\zeta_0) + \cdots
    \end{align}
    as $\zeta\to \zeta_0$, by \ref{lem:elementary:basic}, for some $p_0,p_1,q_0,q_1 \in \C$, then
    \begin{align}
      \label{eqn:q1}
      (p_0^2+p_1)v(\zeta_0) - p_0 q_0 - q_1 = 0.
    \end{align}
  \end{enumerate}
\end{lemma}
\begin{proof}
  The assertion \ref{lem:elementary:basic} follows from the well-known fact that $\operatorname{Res}(fg,\zeta_0) = f(\zeta_0) \operatorname{Res}(g,\zeta_0)$ whenever $f$ is analytic and $g$ has at most a simple pole at $\zeta_0$. If $q \equiv 0$ and $v \not \equiv 0$ then \eqref{eqn:v'+pv=q} rearranges to
  \begin{align*}
    p = -\frac{v_\zeta}v = -\frac{d}{d\zeta} \log v.
  \end{align*}
  A simple calculation shows that $\operatorname{Res}(p,\zeta_0)=-m$, where
  $m \ge 0$ is the order of the zero of $v$ at $\zeta_0$.
  %
  Finally, to see \ref{lem:elementary:q1}, we insert \eqref{eqn:p0p1q0q1} into \eqref{eqn:v'+pv=q} to obtain
  \begin{align*}
    \frac{-2v(\zeta_0)}{\zeta-\zeta_0}
    + (p_0 v(\zeta_0) - v_\zeta(\zeta_0))
    + (p_1 v(\zeta_0) + p_0 v_\zeta(\zeta_0))(\zeta-\zeta_0) + \cdots
      = \frac{-2v(\zeta_0)}{\zeta-\zeta_0} 
      + q_{0} + q_1(\zeta-\zeta_0) +\cdots
  \end{align*}
  as $\zeta\to\zeta_0$, where the terms involving $v_{\zeta\zeta}(\zeta_0)$ cancel because   $\operatorname{Res}(p,\zeta_0)=-2$. Grouping like powers and eliminating $v_\zeta(\zeta_0)$, we obtain \eqref{eqn:q1} as desired.
\end{proof}

\begin{lemma}\label{lem:invertible}
For any $A \in (0,1/2)$, $\L(A) \colon X \to Y$ is invertible.
\end{lemma}

\begin{proof}
  By Lemma~\ref{lem:Fredholm}, it suffices to show that the kernel of $\L(A)$ is trivial. Suppose for the sake of contradiction that 
  $v \not \equiv 0, \in X$ lies in the kernel of $\L(A)$. That is, 
  \[
  \Im\left((1-2\Omega(A))\zeta\frac{(\zeta-A)(\zeta+1/A)}{(\zeta+A)(\zeta-1/A)} v_\zeta-\Omega(A)v+4A\Omega(A)\zeta\frac{v-v(-A)}{(\zeta+A)^2}\right)=0\qquad \text{for $|\zeta|=1$}.
  \]
  Clearly,
  \begin{equation}\label{def:f}
    f:=(1-2\Omega(A))\zeta\frac{(\zeta-A)(\zeta+1/A)}{(\zeta+A)(\zeta-1/A)} v_\zeta-\Omega(A)v+4A\Omega(A)\zeta\frac{v-v(-A)}{(\zeta+A)^2}
  \end{equation}
  is meromorphic in $\D$, possibly with a simple pole at $\zeta=-A$, and $\Im f=0$ on $\partial\D$. 
 
  Since $v(\overline\zeta)=-\overline{v(\zeta)}$ by symmetry, $v$ is purely imaginary on the real axis and, hence, the residue of $f$ at $\zeta=-A$ must be purely imaginary. Therefore
  \[
    f-\frac{if_{-1}}{\zeta+A}
  \]
  is holomorphic in $\D$ for some $f_{-1}\in\R$. Note that 
  \[
     \frac{if_{-1}}{\zeta+A} +\overline{\left(\frac{if_{-1}}{\zeta+A}\right)}
    = \frac{if_{-1}}{\zeta+A}+\frac{if_{-1}/A^2}{\zeta+1/A}-\frac{if_{-1}}A \qquad\text{for $|\zeta|=1$}
  \]
  is real, and we consider the function
  \begin{equation}\label{def:tilde f}
    \tilde f := f-\frac{if_{-1}}{\zeta+A}-\frac{if_{-1}/A^2}{\zeta+1/A}+\frac{if_{-1}}A,
  \end{equation}
  which is not only holomorphic in $\D$ but also real on $\partial\D$. Since $v$ is purely imaginary on $\R \cap \D$, the same is true of $\tilde f$. Together, these lead to  
  \begin{equation}\label{eqn:f=0}
    \tilde f =0
    \qquad \text{in } \D.
  \end{equation}
  Indeed, since $\Im\tilde f = 0$ on $\partial\D$, the maximum principle implies that $\Im \tilde f = 0$ in $\D$. Thus $\tilde f$ is a real and holomorphic function and, hence, a real constant. Finally, since $\tilde f$ is purely imaginary along $\R$, the only possibility is that this constant is $0$.

  Recalling \eqref{def:f} and \eqref{def:tilde f} we write \eqref{eqn:f=0} as the complex ODE in \eqref{eqn:v'+pv=q}, where
  \begin{align*}
    p(\zeta)&=-\frac{\Omega(A)}{1-2\Omega(A)} \frac{(\zeta-A)(\zeta-1/A)}{\zeta(\zeta+A)(\zeta+1/A)}
    = \frac{1-A^2}{1+A^2} \frac 1\zeta - \frac 2{\zeta+A} + \frac 2{\zeta+1/A}
    \intertext{and}
    q(\zeta)&=\frac1{1-2\Omega(A)}\frac{(\zeta+A)(\zeta-1/A)}{\zeta(\zeta-A)(\zeta+1/A)}
    \left(\frac{4A\Omega(A)v(-A)\zeta}{(\zeta+A)^2}
    +\frac{if_{-1}}{\zeta+A}+\frac{if_{-1}/A^2}{\zeta+1/A}-\frac {if_{-1}}A\right)\\
    &= 
    \frac{i(1-3A^2)}{A(1+A^2)} f_{-1} 
    \left( \frac 1\zeta 
    - \frac{(1-A^2)^2}{(1+A^2)^2} \frac 1{\zeta-A}
    - \frac{1+6A^2+A^4}{(1+A^2)^2} \frac 1{\zeta+1/A}
    + \frac{2(1-A^2)}{A(1+A^2)} \frac 1{(\zeta+1/A)^2}
    \right)\\
    &\qquad +
    v(-A) \left(
    \frac{(1-A^2)^2}{(1+A^2)^2} \frac 2{\zeta-A}
    - \frac 2{\zeta+A} 
    + \frac{8A^2}{(1+A^2)^2} \frac 1{\zeta+1/A}\right),
  \end{align*}
  where we use \eqref{def:omega(A)} and partial fractions. Note that potential singularities in $\D$ of $p$ and $q$ are simple poles at $\zeta=\pm A$ and $\zeta=0$.  
  
  At $\zeta=A$ we calculate that
  \[
    \operatorname{Res}(p,A)=0
    \quad\text{and}\quad
    \operatorname{Res}(q,A)=
    2 \frac{(1-A^2)^2}{(1+A^2)^2} v(-A)
    - \frac{i(1-A^2)^2(1-3A^2)}{A(1+A^2)^3} f_{-1}.
  \]
  Applying Lemma~\ref{lem:elementary}\ref{lem:elementary:basic}, we deduce that
  \begin{equation}\label{eqn:f-1}
    f_{-1} = -\frac{2iA(1+A^2)}{1-3A^3} v(-A).
  \end{equation}
  Using this to eliminate $f_{-1}$ in the above formula for $q$, things simplify considerably and we are left with
  \begin{align*}
    q = 2v(-A) \left( \frac 1\zeta - \frac 1{\zeta+A} - \frac 1{\zeta+1/A} 
    + \frac{2A(1-A^2)}{1+A^2} \frac 1{(\zeta+1/A)^2} \right).
  \end{align*}
  In particular, at $\zeta=-A$, we calculate
  \[
    \operatorname{Res}(p,-A)=-2
    \quad\text{and}\quad
    \operatorname{Res}(q,-A)=-2v(-A).
  \]
  Writing 
  \begin{align*}
    p(\zeta) = \frac{-2}{\zeta+A} + p_{0} + p_1(\zeta+A) + \cdots \quad\text{and}\quad 
    q(\zeta) = \frac{-2v(-A)}{\zeta+A} + q_{0} + q_1(\zeta+A) + \cdots
  \end{align*}
  as $\zeta\to -A$, 
  where
  \begin{align*}
    p_0 &= -\frac{1-4A^2-A^4}{A(1-A^2)(1+A^2)},
    &
    p_1 &= -\frac{1-3A^2+5A^4+A^6}{A^2(1-A^2)^2(1+A^2)},
    \\
    q_0 &= -\frac{2v(-A)}{A(1+A^2)},
    &
    q_1 &= -\frac{2(1-A^2+2A^4)v(-A)}{A^2(1-A^2)^2(1+A^2)},
  \end{align*}
  we conclude from Lemma~\ref{lem:elementary}\ref{lem:elementary:q1} that
  \begin{align*}
    (p_0^2+p_1)v(-A) - p_0 q_0 - q_1 
    =
    \frac {4v(-A)}{(1-A^2)^2} = 0.
  \end{align*}
  This forces $v(-A)=0$ and, hence, $f_{-1}=0$ as well by \eqref{eqn:f-1}. Thus $q \equiv 0$ and, therefore, Lemma~\ref{lem:elementary}\ref{lem:elementary:qzero} implies that $\operatorname{Res}(p,0) \le 0$ is an integer. But we calculate
  \[
    \operatorname{Res}(p,0)=\frac{1-A^2}{1+A^2} \in (0,1),
  \]
  which is the desired contradiction.
\end{proof}

\begin{remark*}
Recall $v(-A)=0$ and $f_{-1}=0$ in the course of the proof of Lemma~\ref{lem:invertible}, whence \eqref{eqn:f=0} becomes
\[
  v_\zeta + \frac{1-A^2}{1+A^2} \frac{(\zeta-A)(\zeta-1/A)}{\zeta(\zeta+A)(\zeta+1/A)}v=0,
\]
whose general solution is
\[
  v(\zeta)=C\left(\frac{\zeta+A}{\zeta+1/A}\right)^2\zeta^{\frac{A^2-1}{A^2+1}},
\]
where $C \in \C$ is an arbitrary constant. This is multi-valued in $\D$ unless $A=0$ or $C=0$.
\end{remark*}

\begin{remark*}
An earlier version of this paper contained an error in Lemma~\ref{lem:elementary}, which led to an overly simplistic analysis of the complex ODE in the proof of Lemma~\ref{lem:invertible}. Thankfully, this error was identified by an anonymous referee, who also suggested the current statement and proof of Lemma~\ref{lem:elementary}\ref{lem:elementary:basic}--\ref{lem:elementary:qzero}.
\end{remark*}

For each $A\in(0,1/2)$ and fixed, recall \eqref{eqn:F(A)=0}. We deduce from \eqref{def:F'} and Lemma~\ref{lem:invertible} that 
\[
\text{$\F_w(w(A);0,A)\colon X\to Y$ is invertible.}
\]
Theorem~\ref{thm:IFT} then follows at once from the implicit function theorem for real-analytic operators. 

\begin{proof}[Proof of Theorem~\ref{thm:overhang}]
For $A \in (\sqrt 2 - 1, A_{\max})$ and fixed, Theorem~\ref{thm:IFT} gives a one-parameter family of solutions  $w=W(G,A)$ to \eqref{eqn:w}, where $\Omega=\Omega(A)$ and $B=B(A)$, for $G$ sufficiently small, depending real analytically on $G$. 
Correspondingly, \eqref{def:w} gives a one-parameter family of holomorphic functions $z=Z(G,A)$ whose imaginary part solves \eqref{eqn:bernoulli}. Since $W(G,A) \in U$ by construction, \eqref{eqn:stagnation} holds true, whereby they give rise to physical solutions of \eqref{eqn:stream0} provided \eqref{eqn:injective} holds true.

When $G=0$, recall from Section~\ref{sec2:crapper} that \eqref{eqn:injective} holds true. For $G$ sufficiently small, we consider 
\[
f(\alpha,\alpha') = \frac{z(\alpha)-z(\alpha')}{\alpha-\alpha'},
\qquad
f \colon \T^2 \to \C.
\]
Since $w \in U$, this is well-defined and continuous and, moreover, \eqref{eqn:injective} holds true if and only if $f \ne 0$. Since $\T^2$ is compact, this condition is preserved under small $C^1$ perturbations of $z$ and, hence, small $C^1$ perturbations of $w$.

It remains to show that the solutions indeed give overhanging waves. When $G=0$, we can find $x^* \in \T$ and $\alpha,\alpha' \in \T$ such that $x(\alpha)=x(\alpha')=x^*$, $y_\alpha(\alpha),y_\alpha(\alpha') > 0$, and $x_\alpha(\alpha)<0<x_\alpha(\alpha')$. We then deduce from the implicit function theorem that these conditions continue to hold at some  $\alpha(G),\alpha'(G)\in\T$ for $G$ sufficiently small. This completes the proof.
\end{proof}

\begin{proof}[Proof of Theorem~\ref{thm:touching}]
When $G=0$, the profile of \eqref{def:z(A)} does not intersect itself for $A<A_{\max}$, intersects itself tangentially at one point for $A=A_{\max}$, and intersects itself transversally at two points for $A>A_{\max}$. By continuity, the profile corresponding to the solution $W(G, A_{\max}-\varepsilon/2)$ to \eqref{eqn:w} does not intersect itself for $G$ sufficiently small, while the profile corresponding to $W(G,A_{\max}+\varepsilon/2)$ intersects itself transversally at two points.
We can then find some $A\approx A_{\max}$ such that the profile corresponding to $w(G,A)$ intersects itself tangentially at one point for $G$ sufficiently small. Details are found, for instance, in the proof of \cite[Proposition~10]{CEG:touching}. We remark that this is the only place in this work where the $w \in C^{3+a}(\overline\D)$ is needed. Elsewhere it is sufficient to work with $w \in C^{1+a}(\overline\D)$.

\end{proof}

\section{Extensions}\label{sec:extensions}

There are several directions in which one might take matters further.

\paragraph{Finite depth.} 
In the finite depth setting, we replace \eqref{eqn:infty bc} by
\[
\psi=\text{const.}\qquad\text{on $y=-h$}
\]
for some $h>0$. We can follow along the same line of argument as in Section~\ref{sec:conformal} to arrive at 
\[
(1+\Omega(y+y\Hd_H y_\alpha-\Hd_H(yy_\alpha)))^2=(B-2Gy)((1+\Hd_H y_\alpha)^2+y_\alpha^2)
\qquad \text{for $\beta=0$},
\]
instead of \eqref{eqn:bernoulli}, where
\[
\Hd_H e^{in\alpha}=\begin{cases}
-i\coth(nH)e^{in\alpha}\quad &\text{if $n\neq0$} \\
0 &\text{if $n=0$},
\end{cases}
\]
instead of \eqref{def:H'}. We may assume without loss of generality $H=kh$. See, for instance, \cite{DH1} for details. Since $\Hd_H\to\H$ as $H\to\infty$ (see, for instance, \cite{CSV:constant-vorticity} for details), one can argue as in Section~\ref{sec:IFT} to deduce that for $G$ sufficiently small and $H$ sufficiently large, there exists a solution to \eqref{eqn:stream0}--\eqref{eqn:symm(z)}, replacing the last equation of \eqref{eqn:stream0} by $\psi=\text{const.}$ on $y=-H$, for which the fluid surface does not intersect itself but is overhanging. Also there exists a solution whose fluid surface intersects itself tangentially along the trough line, enclosing a bubble of air, namely a touching wave. We omit the details.

\paragraph{Point vortices.}
When vorticity is concentrated at a point in the fluid region of one period in the moving frame of reference, rather than constant throughout the fluid region, \eqref{def:w(A)} also gives an exact solution for zero gravity (see \cite[Section~4]{CR:pointvortex} for details), whereby one may be able to follow along the same line of argument as in Section~\ref{sec:IFT} for the existence of overhanging and touching waves for sufficiently weak gravity, either in the infinite or finite but sufficiently large depth. 

\paragraph{Hollow vortices.}
Last but not least we consider a hollow vortex, a bounded region of constant pressure with some nonzero circulation around it. Exact solutions have recently been found for rotating hollow vortices with $N$-fold symmetry for any integer $N\geq2$ \cite{CNK;hollow-vortex}. Interestingly, the same conformal mapping also gives exact solutions for non-rotating hollow vortices with the effects of surface tension \cite{crowdy:hollow-vortex, WC:hollow-vortex}. (To compare, Crapper's capillary waves in an irrotational flow give rise to exact solutions in constant vorticity flows without the effects of surface tension.) One then may be able to argue similarly as above for the existence of hollow vortices in fluid regions with sufficiently large but bounded area, either rotating with $N$-fold symmetry or non-rotating with the effects of surface tension.

\section*{Acknowledgments}
The work of VMH was supported by the NSF through the award DMS-2009981. The authors are also indebted to the anonymous referee for their helpful comments, and especially for identifying and resolving an error in an earlier version of Lemma~\ref{lem:elementary}.

\begin{appendix}
\section{Exact solution of \texorpdfstring{\eqref{eqn:stream0}}{Lg} with zero gravity}\label{sec:appn}

Here we detail how to correct errors in \cite{HW:rapids} and show that when $G=0$, \eqref{def:z(A)} and \eqref{def:omega(A)} solve \eqref{eqn:stream0} for $A\in[-A_{\max},A_{\max}]$ (see \eqref{def:Amax}). 
Recall that we identify $\R^2$ with $\C$ and employ the notation $z=x+iy$. Note that \eqref{def:z(A)} extends holomorphically to the lower-half plane as
\begin{equation}\label{def:z(a+ib)}
z(\alpha+i\beta;A)=\alpha+i\beta-\frac{4iAe^{-i(\alpha+i\beta)}}{1+Ae^{-i(\alpha+i\beta)}}.
\end{equation}

We begin by writing
\begin{equation}\label{def:f(appn)}
\psi=-\tfrac12\Omega y^2-y-f,
\end{equation}
so that the first, second and last equations of \eqref{eqn:stream0} become
\begin{equation}\label{eqn:f}
\begin{aligned}
&\nabla^2 f = 0 &&\text{in $D$}, \\
&f=-\tfrac12\Omega y^2-y \quad &&\text{on $S$},\\
&\nabla f\to (0,0) &&\text{as $y\to-\infty$}.
\end{aligned}
\end{equation}
Similarly as in Section~\ref{sec:w(zeta)}, we introduce
\[
\zeta=e^{-i(\alpha+i\beta)},
\]
which takes the values inside the unit disk, and we recast \eqref{eqn:f} in $\D$. (In \cite{HW:rapids}, the incorrect formula $\zeta=e^{i(\alpha+i\beta)}$ was used instead, which takes values outside $\D$. This is the source of most of the errors in \cite{HW:rapids}.) 

The Poisson integral formula gives
\[
f(\zeta)=\text{Re}\left(\frac1{2\pi i}\ointctrclockwise_{|\zeta'|=1} f(\zeta') \frac{\zeta'+\zeta}{\zeta'-\zeta}~\frac{d\zeta'}{\zeta'}\right)\qquad \text{for $|\zeta|<1$},
\]
where 
\[
f(\zeta')=-\frac12\Omega y^2-y
=-\frac{2\Omega(\zeta'^2+2A\zeta'+1)^2}{(\zeta'+A)^2(\zeta'+1/A)^2}
+\frac{2(\zeta'^2+2A\zeta'+1)}{(\zeta'+A)(\zeta'+1/A)}
\qquad\text{for $|\zeta'|=1$}
\]
by \eqref{def:z(A)} (see also \eqref{eqn:Im(w(A))}). This agrees with the corresponding formula in \cite{HW:rapids} except that the latter has a typographical error where the exponent $2$ is missing in the numerators. When $A\in[-A_{\max},A_{\max}]$, since $A_{\max}<1/2$, $f(\zeta')\frac{\zeta'+\zeta}{\zeta'-\zeta}\frac{1}{\zeta'}$ has poles at $\zeta'=-A$, $\zeta$ and $0$ in $\D$, and we evaluate the integral by means of the calculus of residues to arrive at
\[
f(\zeta)=\text{Re}\left(\frac{4\Omega}{A^2-1}\frac{(1-2A^2)\zeta^2+1}{(\zeta+1/A)^2} 
+\frac{4\zeta}{\zeta+1/A}\right)\qquad\text{for $|\zeta|<1$},
\]
so that \eqref{def:f(appn)} becomes
\begin{equation}\label{eqn:psi}
  \psi = 
  - \beta - \frac 12 \Omega y^2 
  - \frac{4\Omega}{A^2-1} \text{Re}\left( \frac{(1-2A^2)\zeta^2+1}{(\zeta+1/A)^2}\right).
\end{equation}
Differentiating \eqref{eqn:psi},
\begin{equation}\label{eqn:psi_beta}
\psi_\beta=-1-\Omega yx_\alpha
-\frac{8\Omega}{A(A^2-1)}\text{Re} \left(\frac{\zeta((1-2A^2)\zeta-A)}{(\zeta+1/A)^3}\right).
\end{equation}
All of these formulas differ from those in \cite{HW:rapids} by $\zeta\mapsto 1/\zeta$. We then use the second equation of \eqref{eqn:stream0} and make a chain rule calculation to see that
\[
|\nabla_{(x,y)}\psi|^2=\frac{\psi_\beta^2}{|z_\alpha|^2}\qquad\text{on $S$}.
\]
Inserting \eqref{eqn:psi_beta} into the third equation of \eqref{eqn:stream0}, after some algebra we arrive at 
\begin{equation}\label{eqn:dynamic'}
\left(1-2\Omega\frac{\zeta^2-2A\frac{1+A^2}{1-A^2}\zeta+1}{(\zeta+A)(\zeta+1/A)}\right)^2
=2B\left(\frac{(\zeta-A)(\zeta-1/A)}{(\zeta+A)(\zeta+1/A)}\right)^2.
\end{equation}
This agrees with \cite[(4.7)]{HW:rapids}, which remains invariant under $\zeta\mapsto1/\zeta$. A straightforward calculation reveals that \eqref{eqn:dynamic'} holds true for \eqref{def:omega(A)}.

\end{appendix}

\bibliographystyle{amsplain}
\bibliography{reference}

\end{document}